\documentclass[11pt,twoside,a4paper]{amsart}

\usepackage{amssymb}
\usepackage{vmargin}
\usepackage{amscd}
\usepackage{stmaryrd}
\usepackage{mathrsfs}
\usepackage[all]{xy}
\usepackage{enumitem}
\usepackage{xr-hyper}
\usepackage{hyperref}
\hypersetup{colorlinks,
allcolors=black}

\setmargins{32mm}{20mm}{14.6cm}{22cm}{1cm}{1cm}{1cm}{1cm}

\newcommand\la{\leftarrow}

\newcommand\id{\mathrm{id}}

\newcommand\ten{\otimes}
\newcommand\hten{\hat{\otimes}}

\renewcommand\H{\mathrm{H}}
\newcommand\z{\mathrm{Z}}
\renewcommand\b{\mathrm{B}}

\newcommand\Q{\mathbb{Q}}

\newcommand\bS{\mathbb{S}}

\newcommand\C{\mathcal{C}}

\newcommand\cD{\mathcal{D}}

\newcommand\cP{\mathcal{P}}
\newcommand\cQ{\mathcal{Q}}

\newcommand\cU{\mathcal{U}}
\newcommand\cV{\mathcal{V}}
\newcommand\cW{\mathcal{W}}

\newcommand\Def{\mathfrak{Def}}
\newcommand\Del{\mathfrak{Del}}

\newcommand\DDel{\uline{\Del}}

\newcommand\m{\mathfrak{m}}

\newcommand\Alg{\mathrm{Alg}}

\newcommand\Hom{\mathrm{Hom}}

\newcommand\map{\mathrm{map}}

\newcommand\ctr{\mathrm{ctr}}

\newcommand\HHom{\underline{\mathrm{Hom}}}
\newcommand\EEnd{\underline{\mathrm{End}}}
\newcommand\DDer{\underline{\mathrm{Der}}}

\newcommand\Aut{\mathrm{Aut}}

\newcommand\cone{\mathrm{cone}}
\newcommand\cocone{\mathrm{cocone}}

\DeclareMathOperator\Ob{\mathrm{Ob}}

\newcommand\Co{\mathrm{Co}}

\newcommand\Art{\mathrm{Art}}

\newcommand\Set{\mathrm{Set}}

\renewcommand\>{\rangle}
\newcommand\Lim{\varprojlim}
\newcommand\LLim{\varinjlim}

\newcommand\into{\hookrightarrow}

\newcommand\xra{\xrightarrow}

\newcommand\rk{\mathrm{rk}}
\newcommand\pr{\mathrm{pr}}

\newcommand\proj{\mathrm{proj}}

\newcommand\bt{\bullet}
\newcommand\by{\times}

\DeclareMathOperator\mc{\mathrm{MC}}
\DeclareMathOperator\mmc{\underline{\mathrm{MC}}}
\DeclareMathOperator\MC{\mathfrak{MC}}

\DeclareMathOperator\Gg{\mathrm{Gg}}

\DeclareMathOperator\ddef{\mathrm{Def}}

\newcommand\diag{\mathrm{diag}\,}

\newcommand\pro{\mathrm{pro}}

\newcommand\half{\frac{1}{2}}

\newcommand\Gpd{\mathrm{Gpd}}

\newcommand\co{\colon\thinspace}

\newcommand\oR{\mathbf{R}}

\newcommand\oL{\mathbf{L}}

\newcommand\uleft\underleftarrow
\newcommand\uline\underline
\newcommand\uright\underrightarrow

\newtheorem{theorem}{Theorem}[section]
\newtheorem{proposition}[theorem]{Proposition}
\newtheorem{corollary}[theorem]{Corollary}

\newtheorem{lemma}[theorem]{Lemma}
\newtheorem*{theorem*}{Theorem}
\newtheorem*{proposition*}{Proposition}
\newtheorem*{corollary*}{Corollary}
\newtheorem*{lemma*}{Lemma}
\newtheorem*{conjecture*}{Conjecture}

\theoremstyle{definition}
\newtheorem{definition}[theorem]{Definition}
\newtheorem*{definition*}{Definition}

\newtheorem*{notation*}{Notation}

\theoremstyle{remark}

\newtheorem{remark}[theorem]{Remark}
\newtheorem{remarks}[theorem]{Remarks}

\newtheorem*{example*}{Example}
\newtheorem*{examples*}{Examples}
\newtheorem*{remark*}{Remark}
\newtheorem*{remarks*}{Remarks}
\newtheorem*{exercise*}{Exercise}
\newtheorem*{property*}{Property}
\newtheorem*{properties*}{Properties}

\begin{document}
 
\begin{abstract}
We show that Hinich's simplicial nerve of the differential graded Lie algebra (DGLA) of derived derivations of a dg algebra $A$ over a dg properad $\cP$ is equivalent to the space of deformations of $A$ as a $\cP_{\infty}$-algebra in Positselski's contraderived dg category. This  resolves Hinich's counterexamples to the general existence of derived deformations. It also generalises his results when $A$ is homologically bounded below, since  contraderived deformations are then precisely derived deformations. 

\end{abstract}

\title[Derived derivations govern contraderived deformations]{Derived derivations govern contraderived deformations of dg algebras over dg (pr)operads}

\author{J. P. Pridham}

\maketitle


\section*{Introduction}

There is a principle, going back to Drinfeld's letter \cite{drinfeldtoschechtman} and implicit in earlier work of Deligne, that any reasonable deformation problem in characteristic $0$ should be governed by a DGLA, a principle  made precise by the equivalence established in \cite{ddt1} (and later \cite{lurieDAG10}) between DGLAs and derived deformations functors satisfying a left-exactness condition (known as homogeneity or cohesiveness). 

In \cite{hinichDefsHtpyAlg}, Hinich proved that derived  deformations of any connective dg algebra $A$  over  a connective operad $\cP$ in characteristic $0$ are governed by its DGLA $\oR\DDer_{\cP}(A,A)$ of derived derivations, in the sense that the $\infty$-groupoid of derived deformations is given by Hinich's simplicial nerve of the DGLA. However, he also gave a simple counterexample showing that the functor of derived  deformations of a non-connective dg algebra is not in general homogeneous, so cannot be governed by any DGLA. 

This raises the question of finding what deformation problem the DGLA of derived derivations does govern in these cases. As is typical for problems involving Koszul duality, the solution is given by Positselski's derived categories of the second kind \cite{positselskiDerivedCategories}, in this case by the contraderived category.
Given a commutative dg Artinian\footnote{Here, it is essential that $R$  be Artinian on the nose, not just up to quasi-isomorphism.}  algebra $R$, the contraderived dg category $\cD_{dg}^{\ctr}(R)$ of $R$ has as objects all $R$-modules $M$ in chain complexes which are quasi-free, i.e.\ free as graded $R$-modules; in particular, if $M$ is not bounded below in the chain direction, this means it need not be cofibrant in the projective model structure. 

We have two equivalent formulations of the resulting deformation problem. The first associates to a commutative dg local Artinian algebra $R$ the space of pairs $(A',\theta)$ where $A'$ is a $\cP$-algebra in  $\cD_{dg}^{\ctr}(R)$ and $\theta \co A'\ten_R k \to A$ a quasi-isomorphism, so $A \simeq A'\ten^{\oL,I\!I}_R k$ (the derived tensor product of the second kind). 
The second formulation instead takes pairs $(A'',\theta)$ where $A''$ is a strong homotopy $\cP$-algebra in $\cD_{dg}^{\ctr}(R)$ and $\theta \co A'\ten_R k \to A$ a quasi-isomorphism. 

Taking the first approach, our main result is Theorem \ref{cofalgdefthm}, which shows that if $A$ is a cofibrant $\cP$-algebra, then the functor sending $R$ to the  $\infty$-groupoid of contraderived deformations $(A',\theta)$ as above is governed by the DGLA $\DDer_{\cP}(A,A)$ of $\cP$-algebra deformations of $A$. This statement requires no connectivity hypotheses on $\cP$, $A$ or $R$, but if $A$ is connective and $R$ eventually connective, then this is also the  $\infty$-groupoid of derived deformations (Corollary \ref{cofalgdefcor}). These results follow directly from an analysis of Maurer--Cartan elements. Corollary \ref{cofalgdefcor} generalises Hinich's results from \cite{hinichDefsHtpyAlg} by weakening the connectivity hypotheses, while Theorem \ref{cofalgdefthm} shows that his counterexamples can be resolved by using contraderived deformations in place of derived deformations.

Taking the second approach, we interpret the space of strong homotopy $\cP$-algebras in $\cD_{dg}^{\ctr}(R)$ as being the space of maps from $\cP$ to $\cD_{dg}^{\ctr,\ten}(R)$ in the $\infty$-category of coloured dg operads (a.k.a. dg multicategories) localised at quasi-equivalences. Corollary \ref{maincor} then shows that space of such  contraderived deformations is governed by a DGLA of derived deformations defined by operadic convolution, and extends the result to algebras over a dg properad (a setting where the first approach is not possible). That corollary also shows that if $A$ is connective and $R$ eventually connective, this agrees with the space of derived deformations, defined using the derived category $\cD_{dg}(R)$ in place of the contraderived category $\cD_{dg}^{\ctr}(R)$.

Corollary \ref{maincor} is a special case of Theorem \ref{mainthm}, which shows that for cofibrant $\cP$,  a DGLA $ \DDer(\widehat{\cP},\cQ)$ governs the space of maps $\cP \to \cQ^{\ctr}(R)$ for a  construction $\cQ \leadsto \cQ^{\ctr}(R)$, defined for any dg properad $\cQ$. 
The $(-)^{\ctr}$ construction  generalises the formation of $\cD^{\ctr}_{dg}(R)$ from $\cD_{dg}(k)$ and $R$,
and is fundamental in the sense that it turns out to be the universal functor under $\cQ\hten-$ with good deformation theory (Remarks \ref{proximateDdgrmk}).
%
%
It
is roughly right adjoint to Chuang and Lazarev's hat construction $\widehat{(-)}$ \cite{ChuangLazarevComb}, 
which
amounts to formally adding a new co-unit to the Koszul dual co(pr)operad. 

When $\cP$ is cofibrant, the  convolution DGLA $\DDer(\widehat{\cP},\cQ)$ is thus the augmented deformation complex as considered in \cite{MerkulovValletteDefThRepsProps}, even though $\widehat{\cP}$ itself is acyclic.  When $\cQ=\cD^{\ten}_{dg}(k)$, this augmented deformation complex  is a model for the DGLA  $\oR\DDer_{\cP}(A,A)$  of derived deformations of a $\cP$-algebra $A$, whereas the unaugmented deformation complex $\DDer(\cP,\cQ)$ often studied in the operadic literature only generalises the cocone of $\oR\DDer_{\cP}(A,A)\to \oR\HHom_k(A,A)$ (Remark \ref{compDGLArmk}).

\smallskip
I would like to thank Andrey Lazarev for helpful comments. I would also like to thank the anonymous referee for identifying errors and suggesting useful references. 

\subsubsection*{Relation with other work}

As described above, \cite{hinichDefsHtpyAlg} addresses the case of a  connective dg algebra over a connective dg operad, and gives a simple counterexample \cite[Example 4.3]{hinichDefsHtpyAlg} to the possibility of derived deformations of algebras being governed by DGLAs in general.
For  deformations of modules, which correspond to the special case of $\cP$ being an associative dg algebra (or a dg category if we incorporate colours), \cite[Proposition 6.1 and Theorem 11.6]{EfimovLuntsOrlov1}  give the $1$-categorical truncations of Theorem \ref{cofalgdefthm} and Corollary \ref{cofalgdefcor}  (via Lemma \ref{modQIMlemma}).

There is a significant literature working with mapping spaces of dg operads rather than deformations of cofibrant models, going back to \cite[\S 4]{rezkThesis} and allowing a generalisation to properads, but they  work with monochrome dg (pr)operads, implicitly studying deformations which fix the underlying chain complex. In particular, \cite[Theorem 12.2.4]{LodayValletteOperads}
shows that $\ker(\DDer_{\cP_{\infty}}(A,A)\to \HHom_k(A,A))$ gives a DGLA governing such deformations. However,  that DGLA seldom has finite-dimensional cohomology groups for examples of interest, and in applications one usually wishes to allow the  complex underlying a dg algebra to deform non-trivially.

Allowing the chain complex to vary in particular leads to quasi-isomorphism rather than $\infty$-isotopy 
as the notion of equivalence. That perspective is taken up in \cite{GinotYalin}, but for these deformations they focus on  derived automorphisms of the trivial deformation without investigating the possible ways the algebraic structure can deform. A universal property follows for the functor governed by the DGLA of derived derivations \cite[Remark 2.26 and Theorem 2.28]{GinotYalin}, without determining whether or how its elements correspond to derived deformations. It is a consequence of our results that the universal map in question must be that from derived deformations to contraderived deformations, and that the map is an equivalence for eventually connective dg algebras (Remarks \ref{1proxrmk1} and \ref{proximateDdgrmk}).

\tableofcontents

\subsubsection*{Notation}

For a chain (resp. cochain) complex $M$, we write $M_{[i]}$ (resp. $M^{[j]}$) for the complex $(M_{[i]})_m= M_{i+m}$ (resp. $(M^{[j]})^m = M^{j+m}$).  When we  need to compare chain and cochain complexes, we silently  make use of the equivalence  $u$ from chain complexes to cochain complexes given by $(uV)^i := V_{-i}$, and refer to this as rewriting the chain complex as a cochain complex (or vice versa). On suspensions, this has the effect that $u(V_{[n]}) = (uV)^{[-n]}$. We also occasionally write $M[i]:=M^{[i]} =M_{[-i]}$ when there is only one grading. 

We use the subscripts and superscripts $\bt$ to emphasise that chain and cochain complexes incorporate differentials, with $\#$ used instead when we are working  with the underlying graded objects.

Given $A$-modules $M,N$ in chain complexes, we write $\HHom_A(M,N)$ for the cochain complex given by
\[
 \HHom_A(M,N)^i= \Hom_{A_{\#}}(M_{\#},N_{\#[-i]}),
\]
with differential $d f= d_N \circ f \mp f \circ d_M$,
where $V_{\#}$ denotes the graded vector space underlying a chain complex $V$, and $\mp$ the Koszul sign.

Given chain complexes $U$ and $V$ over $k$, we write $U\ten_k V$ for the chain complex given by the (direct sum) total complex of the external tensor product, so  $(U\ten_k V)_n= \bigoplus_{i+j=n} U_i\ten V_j$ with differential $d_U\ten \id \mp \id\ten d_V$.  Given a pro-object $V= \{V(\alpha)\}_{\alpha}$ in chain complexes, we write $U\hten V$ for the completed tensor product $\Lim_{\alpha}(U\ten V(\alpha))$.

We write $s\Set$ for the category of simplicial sets, and $\oR\map$ for derived mapping spaces, i.e.\ the right-derived functor of  $\Hom$ regarded as a simplicial set-valued bifunctor. 

\section{Background results} 


Fix a field $k$ of characteristic $0$.

\subsection{Derived deformation functors}

\subsubsection{Artinian cdgas and DGLAs} 

\begin{definition}\label{cdgadef}
A cdga (commutative differential graded algebra) $R_{\bt}$ over $k$ is a chain complex of $k$-vector spaces equipped with a   unital associative graded-commutative  multiplication with respect to which the differential acts as a derivation. Say that $R_{\bt}$ is local Artinian if it admits a $k$-cdga homomorphism $R_{\bt} \to k$ for which the kernel $\m(R_{\bt})$ is nilpotent and finite-dimensional. 

Write $dg\Art_k$ for the category of local Artinian $k$-cdgas, $dg_+\Art_k \subset dg\Art$ for the full subcategory of non-negatively graded objects  $\ldots \to R_2 \to R_1 \to R_0$, and $\Art_k \subset dg_+\Art_k$ for the full subcategory on objects concentrated in degree $0$.
\end{definition}

\begin{definition}
 A differential graded Lie algebra (DGLA) $L^{\bt}$ over $k$ is a cochain complex  of $k$-vector spaces equipped with a graded-Lie bracket with respect to which the differential acts as a derivation.
\end{definition}
%
%

\subsubsection{Extended functors associated to DGLAs}

\begin{definition}\label{MCdef}
 Given a DGLA $L$, the Maurer--Cartan set $\mc(L)$ is defined by 
 \[
  \mc(L):=\{\omega \in L^1 ~:~ d\omega +\half [\omega,\omega]=0\}.
 \]

 If $L^0$ is nilpotent, define the gauge group $\Gg(L)$ to consist of grouplike elements in the completed universal enveloping algebra $\widehat{\cU}(L^0)$ (a complete Hopf algebra). The exponential map gives an isomorphism to $\Gg(L)$ from the set $L^0$ equipped with the Campbell--Baker--Hausdorff product.
 
 When the action of $L^0$ on $L^1$ is moreover nilpotent, there is a gauge action of $\Gg(L)$ on $\mc(L)$, given by $g\star \omega:= g\omega g^{-1} -(dg)g^{-1}$ (evaluated in $\widehat{\cU}(L)^1$); see \cite[\S 1]{Man} or \cite[Lecture 3]{Kon}.

 Denote the quotient set $\mc(L)/\Gg(L)$ by $\ddef(L)$, and the quotient groupoid $[\mc(L)/\Gg(L)]$ by $\Del(L)$ (the Deligne groupoid).
 \end{definition}
The terminology has its origins in constructions associated to the DGLA of differential forms valued in an adjoint bundle, where the Maurer--Cartan equation parametrises flat connections and the gauge action corresponds to gauge transformations. 
 
 \subsubsection{Hinich's simplicial nerve and variants}
 
 The following, when restricted to $dg_+\Art$, is Hinich's nerve $\Sigma_L$ from \cite[Definition 8.1.1]{hinstack}.
 \begin{definition}\label{mmcdef}
  Given a DGLA $L$, define the simplicial set-valued functor $\mmc(L,-)\co \pro(dg\Art) \to s\Set$ by 
  \[
   \mmc(L,R)_n:= \mc((L\ten \Omega^{\bt}(\Delta^n))\hten \m(R)), 
  \]
  with the obvious simplicial structure maps, where $\Omega^{\bt}(\Delta^n)$ is the cdga $\Q[t_0, \ldots, t_n, dt_0, \ldots, dt_n]/(\sum t_i-1,\, \sum dt_i)$ of de Rham polynomial forms on the $n$-simplex, with $t_i$ of degree $0$. 
  
  We write $\mc(L,R)$ for  $\mmc(L,R)_0$.
 \end{definition}

 We also have the following variants:
\newcommand\DDEL{\uline{\mathfrak{DEL}}}
\begin{definition}\label{ddeldef}
  Given a DGLA $L$, define the simplicial groupoid-valued functor $\DDEL(L,-)\co \pro(dg\Art) \to \Gpd^{\Delta}$ by 
  \[
   \DDEL(L,R)_n:= \Del(L\ten \Omega^{\bt}(\Delta^n),R),
  \]
 and the functor $\DDel(L,-)\co dg\Art \to s\Gpd$, taking values in simplicially enriched groupoids, by letting $\DDel(L,R)$ have objects $\mc(L,R)$ and simplicial sets $ \DDel(L,R)(\omega,\omega')$ of morphisms given by
 \[
  \DDel(L,R)(\omega,\omega')_n:=\{g \in \Gg((L\ten \Omega^{\bt}(\Delta^n))\hten \m(R))~:~ g \star \omega = \omega' \in  \mmc(L,R)_n\}.
 \]
 \end{definition}

 Given a simplicial groupoid $\Gamma$, we can apply the nerve construction $B$ to give a bisimplicial set $B\Gamma$, then take the diagonal to give a simplicial set $\diag B \Gamma$. 
 The following lemma then gives us natural alternative interpretations of the simplicial nerve.
 
\begin{lemma}\label{MCDelequivlemma}\cite[Lemma \ref{utrecht-MCDelequivlemma}]{utrecht}
There are natural weak equivalences
 \[
  \mmc(L,R) \to \diag B \DDEL(L,R) \la \diag B\DDel(L,R)
 \]
for all DGLAs $L$ and local Artinian cdgas $R$. 
\end{lemma}

\subsection{Derived and contraderived dg categories}

\begin{definition}
Given $R \in dg\Art_k$, define the derived dg category $\cD_{dg}(R)$ to consist of all cofibrant  objects in the projective model structure on  $R$-modules in complexes (e.g. \cite[Theorem 8.1a]{positselskiDerivedCategories}), with morphisms given by the usual cochain  complexes $\HHom_R(M,N)$ of graded $R$-module homomorphisms. 
\end{definition}
Since $R$ is Artinian, all projective modules are free, so this means that objects $M$ of $\cD_{dg}(R)$ are free as graded $R$-modules, with the additional condition that $M$ must admit an expression as a filtered colimit $\LLim_{i \in I}M(i)$ of $R$-subcomplexes indexed by some ordinal $I$, with each $M(i+1)/M(i)$ isomorphic as an $R$-complex to a direct sum of shifts of copies of $R$.

Since all objects are fibrant in the projective model structure, the category $\H^0\cD_{dg}(R)$ (with the same objects, but morphisms $\H^0\HHom_R(M,N)$) is then just the homotopy category for the model structure, i.e.\ the derived category $\cD(R)$.

\begin{definition}
  Given $R \in dg\Art_k$, define the contraderived dg category $\cD^{\ctr}_{dg}(R)$ to have as objects $R$-modules $M$ in chain complexes for which the underlying graded $R_{\#}$-module $M_{\#}$ is free (equivalently, projective). Morphisms are given by the usual cochain complexes $\HHom_R(M,N)$ of graded $R$-module homomorphisms.
   
 More generally, given   $R = \{R(i)\}_i \in \pro(dg\Art_k)$, define $\cD^{\ctr}_{dg}(R)$ to be the $2$-categorical limit of the dg categories $\cD^{\ctr}_{dg}(R(i))$ with respect to the tensor product functors $\ten_{R(i)}R(j)$. Thus every object $M=\{M(i)\}_i$ is isomorphic to one for which the underlying graded $R_{\#}$-module $M_{\#}$ takes  the form 
 $\{ R(i)_{\#}\ten V_{\#}\}_i$ for some graded $k$-vector space $V_{\#}$, with morphisms  given by $\HHom_R(M,N):=  \Lim_i \HHom_{R(i)}(M(i),N(i))$.

 \end{definition}
%
%
 %
In the notation of \cite{positselskiDerivedCategories}, $\cD^{\ctr}_{dg}(R)$ is denoted $\mathrm{DG}(R-\mathrm{mod}_{\proj})$ for $R \in dg\Art_k$, and  $\mathrm{DG}(R^{\vee}-\mathrm{contra}_{\proj})$ for $R  \in \pro(dg\Art_k)$, where $R= \{R(i)\}_i$ and $R^{\vee}$ is  the dual dg coalgebra $\LLim_i \HHom_k(R(i),k)$.  

By   \cite[Theorems 3.8 and 4.4d]{positselskiDerivedCategories}, the homotopy category $\H^0\cD^{\ctr}_{dg}(R)$ is equivalent to the contraderived category $\cD^{\ctr}(R)$ of $R$ when $R$ is Artinian  and to the  contraderived category $\cD^{\ctr}(R^{\vee})$ of the dg coalgebra $R^{\vee}$ when $R$ is pro-Artinian. The contraderived category is given by localising the category of $R$-modules in complexes at maps whose cones are contra-acyclic, meaning that they lie in the minimal triangulated subcategory containing total modules of exact triples and closed under infinite products. 

The objects of $\cD^{\ctr}_{dg}(R)$ correspond to the cofibrant objects in the model category ``of the second kind'' from \cite[Theorem 8.2b]{positselskiDerivedCategories}, in which fibrations are surjections. Since this model structure has more cofibrations and fewer weak equivalences than the standard projective model structure (``of the first kind''), but the same fibrations, the identity functor from the model structure of the second kind to that of the first kind is right Quillen. 

The inclusion $\cD_{dg}(R) \into \cD^{\ctr}_{dg}(R)$ corresponds to its left adjoint restricted to cofibrant objects. It thus has a  derived right adjoint, which is given by cofibrant replacement in the model structure of the first kind (a bar construction), so $\cD_{dg}(R)$ is a right admissible dg subcategory of  $\cD^{\ctr}_{dg}(R)$.

 \begin{lemma}\label{connectivitylemma}
 For $R \in dg_+\Art$, if $M \in \cD^{\ctr}_{dg}(R)$ has $\H_i(M\ten_Rk)=0$ for all $i \ll 0$, then $M \in \cD_{dg}(R)$. 
 \end{lemma}
\begin{proof}
 We need to show that $M$ is cofibrant in the projective model structure of the first kind. Since $k$ is a field, we can decompose $M\ten_Rk$ as $V \oplus U \oplus d(U)$, with $d$ injective on $U$ and zero on $V$. The hypothesis ensures there exists $n_0$ such that $V_i=0$ for all $i < n_0$.
 
 Now, choose lifts $\tilde{V}$ and $\tilde{U}$ of $V$ and $U$ to $k$-linear graded subspaces of $M$. We can then express $M$ as the colimit of the diagram
 \[
  Rd(\tilde{U}) \into R\tilde{U} \oplus Rd(\tilde{U}) \into R\tilde{U} \oplus Rd(\tilde{U})\oplus R\tilde{V}_{n_0} \into R\tilde{U}\oplus Rd(\tilde{U})\oplus R\tilde{V}_{n_0} \oplus R\tilde{V}_{n_0+1 } \into \ldots
 \]
These are subcomplexes, with $d=0$ on the quotients because $d(\tilde{V}_n) \subset R\tilde{U} \oplus Rd(\tilde{U})\oplus \bigoplus_{n_0\le i < n}R\tilde{V}_i$ for degree reasons, so $M$ is cofibrant.
 \end{proof}

\begin{definition}
 For $R \in dg\Art$, define the coloured (symmetric) dg operads (a.k.a. symmetric multicategories) $\cD^{\ten}_{dg}(R)$ and $\cD^{\ctr,\ten}_{dg}(R)$ to be those associated to the tensor products $\ten_R$ on the dg categories $\cD_{dg}(R)$ and $\cD^{\ctr}_{dg}(R)$. 
 
 Explicitly, for $M_1, \ldots, M_r, N \in \cD^{\ctr}_{dg}(R)$, we have
 \[
  \cD^{\ctr,\ten}_{dg}(R)(M_1,\ldots, M_r; N):=\HHom_R(M_1\ten_RM_2\ten_R\ldots\ten_RM_r;N)
 \]
with the obvious composition rules, with $\cD^{\ten}_{dg}(R)$ the full dg suboperad on colours $\Ob \cD_{dg}(R)$.

For $R =\{R(i)\}_i\in \pro(dg\Art)$, define  $\cD^{\ctr,\ten}_{dg}(R)$  by taking limits, so for $M_r= \{M_r(i)\}_i$ and $N=\{N(i)\}_i$ in $\cD^{\ctr}_{dg}(R)$, we have
\[
  \cD^{\ctr,\ten}_{dg}(R)(M_1,\ldots, M_r; N):=\Lim_i\cD^{\ctr,\ten}_{dg}(R(i))(M_1(i),\ldots, M_r(i); N(i)).
 \]
\end{definition}
 
  Given a morphism $f\co R \to S$ in $\pro(dg\Art)$, there is a natural dg tensor functor $\ten_RS \co \cD^{\ctr}_{dg}(R) \to \cD^{\ctr}_{dg}(S) $, and hence a multifunctor  $\cD^{\ctr,\ten}_{dg}(R) \to \cD^{\ctr,\ten}_{dg}(S) $. Explicitly, we can always write $f$ as a filtered limit of maps $f_i \co R_i \to S_i$ in $dg\Art$ indexed by an ordinal $I$ \cite[Appendix, Corollary 3.2]{arma}, and then we set $(M\ten_RS)(i):= M(i)\ten_{R(i)}S(i)$.

  \begin{definition}
 Similarly, define  the coloured dg properads  $\cD^{\ten,\ten}_{dg}(R)$ and  $\cD^{\ctr,\ten,\ten}_{dg}(R)$, of cofibrant $R$-modules and  projective $R$-contramodules respectively, to have the same objects as  $\cD_{dg}(R)$ and $\cD^{\ctr}_{dg}(R)$, and morphisms
 \[
  \cD^{\ctr,\ten,\ten}_{dg}(R)(M_1,\ldots, M_r; N_1, \ldots, N_s):=\HHom_R(M_1\ten_RM_2\ten_R\ldots\ten_RM_r;N_1\ten_RN_2\ten_R\ldots\ten_RN_s)
 \]
with the obvious composition rules, and $\cD^{\ten,\ten}_{dg}(R)$ the full dg subproperad.
  \end{definition}
  
%
%
\subsection{The model structure on coloured dg properads}

Whereas the results so far feature throughout the rest  of the paper,  we now establish a result we will only use in \S \ref{propsn}. 

By \cite{tabuadaMCdgcat}, there is a model structure on the category of small $k$-linear dg categories in which weak equivalences are given by quasi-equivalences, i.e.\
dg functors $F \co \C \to \cD$ 
which are componentwise quasi-isomorphisms 
with $\H^0F \co \H^0\C \to \H^0\cD$  an equivalence of categories. Here, the homotopy category $\H^0\C$ has the same objects as $\C$, but morphisms $\H^0\C(x,y)$ from $x$ to $y$. Fibrations in this model structure are componentwise surjections for which $\H^0F$ is an isofibration of categories in the sense that for any object $x \in \H^0\C$ and any isomorphism $v \co F(x) \to y$ in $\H^0\cD$, there exists an isomorphism $\tilde{v} \co x \to \tilde{y}$  with $F(\tilde{v})=v$.

In \cite[Theorem 5.2]{stanculescuModel} (taking $\cV$ to be $k$-linear chain complexes), this is extended to give a model structure on the category of  small $k$-linear symmetric dg multicategories (i.e.\ coloured $k$-linear dg operads) in which weak equivalences (resp. fibrations) are componentwise quasi-isomorphisms (resp. surjections) $F$ for which $\H^0F$ is an equivalence (resp. isofibration) on the underlying categories.

The most general form we will need is the following analogous statement for coloured dg properads, proved in \cite[Theorem 6.5.5 and Example 6.5.6]{yauDKHtpyTh}. The weak equivalences for this model structure will be referred to as quasi-equivalences.

\begin{proposition}\label{properadmodellemma}
There is a cofibrantly generated model structure on the category of small coloured $k$-linear dg properads in which weak equivalences (resp. fibrations) are componentwise quasi-isomorphisms (resp. surjections) $F$ for which $\H^0F$ induces an equivalence (resp. isofibration) on the underlying categories. 
\end{proposition}

\section{Derived and contraderived deformations of algebras over dg operads}\label{catPalgsn}

 
 \begin{definition}
Given a $k$-linear dg operad $\cP$, let $ \Alg_{\cP}^{\ctr}(R) $ (resp. $ \Alg_{\cP}(R) $) be the category whose objects are $\cP$-algebras in $\cD^{\ctr,\ten}_{dg}(R)$ (resp. $\cD^{\ten}_{dg}(R)$)  and whose morphisms are elements of $\z^0\HHom_{R}$ respecting the $\cP$-algebra structures.

 \end{definition}

 \begin{definition}\label{DDerdef}
Given a  $\cP$-algebra $A$ in a pre-triangulated dg tensor category $\C$, we write $\DDer_{\cP}(A,M)$ for the complex of  $\cP$-derivations from $A$ to $M$.
Here, $M$ is a Beck $A$-module, meaning that $A \oplus M$ is a   $\cP$-algebra in $\C$ for which the projection map $A \oplus M \to A$ and the addition map $(A\oplus M)\by_A(A \oplus M) \to A \oplus M$ are both $\cP$-algebra homomorphisms. Explicitly, 
\[
 \DDer_{\cP}(A,M)^n:= \Hom(A, A \oplus \cone(M)_{[-n]} )\by_{\Hom(A, A)}\{\id\},
\]
where $\Hom$s on the right-hand side are in the category of $\cP$-algebras in $\C$.
The differential $\DDer_{\cP}(A,M)^n\to \DDer_{\cP}(A,M)^{n+1}$ is then induced by the obvious map $\cone(M)_{[-n]} \to \cone(M)_{[-n-1]} $.
 \end{definition}
 
 \begin{definition}
  We say that a morphism $R \to R/I$ in $dg\Art$ is a small extension if $I\cdot \m(R)=0$.
 \end{definition}

 Writing $\bar{A}:= A/\m(R)A$ and similarly for $C$, we have:
 \begin{lemma}\label{defHomlemma} 
 Take $\cP$-algebras $C,A \in \Alg_{\cP}^{\ctr}(R)$, a small extension $R \to R/I$  and a  morphism $f \co \bar{C} \to \bar{A}$ in $\Alg_{\cP}(k)$, with $\bar{C}$ cofibrant in the model structure of \cite[Theorem 4.1.1]{hinichHtpyAlg}. We then have a short exact sequence
 \begin{align*}
  0 \to \z^0 \DDer_{\cP}(\bar{C} ,f_* \bar{A}\ten_k I)
  \to &\Hom_{\cP\ten R}(C, A)_f \\
  &\to \Hom_{\cP\ten (R/I)}(C/IC, A/IA)_f \xra{o_f} \H^1 \DDer_{\cP}(\bar{C},f_* \bar{A}\ten_k I)
 \end{align*}
of groups and sets, where $(-)_f$ denotes the fibre over $ f \in \Hom_{\cP}(\bar{C}, \bar{A})$.
 \end{lemma}
\begin{proof}
 This is a standard obstruction theory argument. Since $\bar{C}$ is cofibrant, $C_{\#}$ is a retract of a freely generated graded $\cP_{\#}\ten R_{\#}$-algebra. We can therefore lift $f' \in \Hom_{\cP\ten (R/I)}(C/IC, A/IA)_f$  to a morphism $\tilde{f} \co C_{\#} \to A_{\#}$ of graded algebras, with any other choice being of the form $\tilde{f}+u$ for a $\cP_{\#}\ten R_{\#}$-derivation $u \co C_{\#} \to IA_{\#}$, or equivalently a $\cP_{\#}$-derivation $ \bar{C}_{\#} \to I\ten_k \bar{A}_{\#}$,   since $I\cdot \m(R)=0$. 
 
 The obstruction to this being a chain map is the commutator $[d, \tilde{f}]$, which is a closed  $\cP\ten R$-derivation $C \to IA$ of cochain degree $1$, or equivalently a $\cP$-derivation $C \to I\ten_k \bar{A}$. A different choice adds $[d,u]$ so exactness on the right follows by setting $o_f(f')$ to be the class of $[d, \tilde{f}]$.
 
 When the lift $\tilde{f}$ is closed, an alternative lift $\tilde{f}+u$ is closed if and only if $u$ is so, giving exactness on the left.
\end{proof}

 \begin{lemma}\label{defOblemma} 
 Take a small extension $R \to R/I$ and  $C \in \Alg_{\cP}^{\ctr}(R/I)$ with $\bar{C}\in \Alg_{\cP}(k)$ cofibrant. Then the  potential  obstruction to lifting $C$ to
 $\Alg_{\cP}^{\ctr}(R)$ lies in $\H^{2} \DDer_{\cP}(\bar{C},\bar{C}\ten_k I)$.
 
 If the obstruction vanishes, the set of lifts is a torsor for $\z^1 \DDer_{\cP}(\bar{C},\bar{C}\ten_k I)$
 \end{lemma}
\begin{proof}
 This is another standard obstruction theory argument. Since $\bar{C}$ is cofibrant, $C$ is a retract of a freely generated graded algebra, so  we can lift $C$ to a graded $\cP_{\#}\ten R_{\#}$-algebra $\tilde{C}_{\#}$ which is a retract of a free algebra. Freeness also allows us to lift the differential $d$ on $C$ to a differential $\tilde{d}$ on $\tilde{C}_{\#}$, with alternative choices given by $\tilde{d}+v$ for graded $\cP_{\#}\ten R_{\#}$ derivations $v \co \tilde{C}_{\#} \to I\tilde{C}_{\#[-1]}$, or equivalently $\cP$-derivations $v \co \bar{C}_{\#} \to I\ten_k\bar{C}_{\#[-1]}$.
 
 The data $(\tilde{C}_{\#},\tilde{d})$ define an object of $\cD^{\ctr,\ten}_{dg}(R)$ if and only if $\tilde{d}\circ \tilde{d}=0$. Since $(\tilde{d}+v)\circ (\tilde{d}+v) =\tilde{d}\circ \tilde{d} +[d,v]$ in $\z^{2} \DDer_{\cP}(\bar{C},\bar{C}\ten_k I)$, it follows that the obstruction is the class of $\tilde{d}\circ \tilde{d}$ in homology.  The choices of $v$ which do not affect the obstruction are precisely those with $[d,v]=0$, giving the torsor statement.
 \end{proof}
 
 \begin{definition}
  Say that a morphism $f \co C \to A$   in $\Alg_{\cP}^{\ctr}(R)$ is a mod-cofibration (resp. mod-quasi-isomorphism) if $\bar{f} \co C/\m(R)C \to A/\m(R)A$ is a cofibration (resp. quasi-isomorphism) in  the standard model structure on $\Alg_{\cP}(k)$.
 \end{definition}

 \begin{lemma}\label{modQIMlemma}
 For a morphism $f \co C \to A$   in $\Alg_{\cP}^{\ctr}(R)$, the following conditions are equivalent:
 \begin{enumerate}
  \item  the cone of $f$ is contra-acyclic,
  \item the morphism 
  underlying $f$ is a  homotopy equivalence in $\cD^{\ctr}_{dg}(R)$,
  \item $f$ is a mod-quasi-isomorphism.
 \end{enumerate}
\end{lemma}
\begin{proof}
 The conditions all remain unchanged if we take $\cP$ to be the trivial operad $I$ and take $C=0$, replacing $A$ with $\cone(f)$. Note that  $A$ is cofibrant in the model structure of  \cite[Theorem 8.2b]{positselskiDerivedCategories}. When  $A$ is contra-acyclic, it is   trivially cofibrant, so   the fibration (i.e. surjection)  $\cocone(A) \to A$  must admit a section, giving the implication (1)$\implies$(2).  If $A$ has a contracting homotopy, then $A/\m(R)A$ also does, giving $\H_*(A/\m(R)A)\cong 0$, so (2)$\implies$(3).
  Finally, since $A_{\#}$ is projective, 
  $A$   has the left lifting property with respect to all ﬁbrations with kernel $M$ whenever $\H^1\HHom_R(A,M)=0$. If $\H_*(A/\m(R)A)\cong 0$,
  this happens for all $M$ because $ \HHom_{R}(A,\m(R)^nM/\m(R)^{n+1}M)\simeq 0$ for all $n$, so $A$ is trivially coﬁbrant and (3)$\implies$(1).
\end{proof}

 \begin{proposition}\label{AlgPlocprop}
  For any $k$-linear  dg operad $\cP$ and any $R \in dg\Art_k$, the simplicial localisation  of $\Alg_{\cP}^{\ctr}(R)$ at 
  mod-quasi-isomorphisms  is equivalent to the simplicial category $\uline{\Alg}^{\ctr,c}_{\cP}(R)$  whose objects are mod-cofibrant $\cP$-algebras in $\cD^{\ctr,\ten}_{dg}(R)$, with morphisms
  \[
  n \mapsto  \Hom_{\cP\ten R\ten \Omega^{\bt}(\Delta^n)}(A\ten \Omega^{\bt}(\Delta^n) , A'\ten \Omega^{\bt}(\Delta^n)) \cong \Hom_{\cP\ten R}(A, A'\ten \Omega^{\bt}(\Delta^n))
  \]
in simplicial level $n$.
   \end{proposition}
\begin{proof}
For any $A \in \Alg_{\cP}^{\ctr}(R)$,  there exists a Reedy cofibrant cosimplicial frame  $\breve{\bar{A}}^{\ast}$ of the $\cP$-algebra  $\bar{A}= A/\m(R)A$ in  $\cD^{\ten}_{dg}(k)$, as in   \cite[\S 5.4]{hovey}. Working up the tower $R/\m(R)^{n+1} \to R/\m(R)^n$ of small extensions, we can inductively  apply Lemmas \ref{defHomlemma} and \ref{defOblemma} to  construct a cosimplicial $\cP$-algebra  $\breve{A}^{\ast}$ in  $\cD^{\ctr,\ten}_{dg}(R)$ lifting $\breve{\bar{A}}^{\ast}$, equipped with a map $\breve{A}^0 \to A$ lifting $\breve{\bar{A}}^0 \to \bar{A}$.

An application of the same tower and lemmas shows that the simplicial functor $n \mapsto \Hom_{\cP\ten R}(\breve{A}^n,-)$ sends mod-quasi-isomorphisms in
$\Alg_{\cP}^{\ctr}(R)$ to weak equivalences. It also shows that  the simplicial bifunctor $n \mapsto \Hom_{\cP\ten R}(-, -\ten \Omega^{\bt}(\Delta^n))$ sends mod-quasi-isomorphisms in either input to weak equivalences, provided we restrict to mod-cofibrant objects on the left. 

 The result now follows from \cite{DKEquivsHtpyDiagrams} by an identical argument to the last part of \cite[Corollary \ref{DQpoisson-Tatehomcor}]{DQpoisson}.
%
\end{proof}

\begin{definition}
Given $R \in dg\Art_k$ and $A \in \Alg_{\cP}(k)$, define the simplicial category of derived deformations of $A$ over $R$ to be the homotopy fibre of the simplicial functor
 \[
  \oL^{\cW} \Alg_{\cP}(R) \to \oL^{\cW} \Alg_{\cP}(k),
 \]
where $\oL^{\cW}$ denotes simplicial localisation at the class $\cW$ of quasi-isomorphisms.

Given $R \in \pro(dg\Art_k)$ and $A \in \Alg_{\cP}(k)$, define the simplicial category of contraderived deformations of $A$ over $R$ to be the homotopy fibre of the simplicial functor
 \[
  \oL^{\bar{\cW}} \Alg^{\ctr}_{\cP}(R) \to \oL^{\cW} \Alg_{\cP}(k),
 \]
where $\oL^{\bar{\cW}}$ denotes simplicial localisation at the class $\bar{\cW}$ of mod-quasi-isomorphisms.
\end{definition}

\begin{lemma}\label{2fibrnlemma}
 For any  small extension $R \to R/I$, the simplicial functor $\uline{\Alg}^{\ctr,c}_{\cP}(R) \to \uline{\Alg}^{\ctr,c}_{\cP}(R/I)$ is a 2-fibration of simplicial categories in the sense of   \cite[Definition 2.22]{dmsch}.
\end{lemma}
\begin{proof}
This follows by exactly the same argument as \cite[Lemma 4.2.1]{hinichDefsHtpyAlg}. Firstly, given $C,A \in \Alg^{\ctr,c}_{\cP}(R)$, we can see that the map
\[
 \uline{\Alg}^{\ctr,c}_{\cP}(R)(C,A) \to  \uline{\Alg}^{\ctr,c}_{\cP}(R/I)(C/IC,A/IC)
\]
is a Kan fibration of simplicial sets by applying Lemma \ref{defHomlemma} to its partial matching maps.

Secondly, to see that $\pi_0\uline{\Alg}^{\ctr,c}_{\cP}(R) \to \pi_0\uline{\Alg}^{\ctr,c}_{\cP}(R/I)$ is an isofibration of categories, note that Proposition \ref{AlgPlocprop} implies that homotopy equivalences in $\uline{\Alg}^{\ctr,c}_{\cP}(R)$ are precisely mod-quasi-isomorphisms. Given $C \in \uline{\Alg}^{\ctr,c}_{\cP}(R)$ and a mod-quasi-isomorphism $f \co C/IC \to A$ in $\uline{\Alg}^{\ctr,c}_{\cP}(R/I)$, we thus seek a lift $\tilde{f} \co C \to \tilde{A}$.  Combining Lemmas \ref{defHomlemma} and \ref{defOblemma}, we see that the potential obstruction to such a lift lies in 
\[
 \H^1\DDer_{\cP}(\bar{C}, \cone(\bar{C} \to \bar{A})),
\]
for $\bar{C}:= C\ten_Rk$ and $\bar{A}:= A\ten_{R/I}k$. Since $f$ is a mod-quasi-isomorphism, the cone is acyclic and the obstruction vanishes.
\end{proof}

\begin{theorem}\label{cofalgdefthm}
Given a $k$-linear dg operad $\cP$ and a $\cP$-algebra $A$ with cofibrant replacement $A'$, the derived Deligne groupoid $\DDel(\DDer_{\cP}(A',A'), R)$ is canonically quasi-equivalent to the simplicial category of contraderived deformations of $A$ over $R$, for all $R \in dg\Art_k$.
\end{theorem}
\begin{proof}
Proposition \ref{AlgPlocprop} allows us to replace $\oL^{\bar{\cW}} \Alg^{\ctr}_{\cP}(R)$ with $\uline{\Alg}^{\ctr,c}_{\cP}(R)$ and $\oL^{\cW} \Alg^{\ctr}_{\cP}(k)$ with $\uline{\Alg}^{c}_{\cP}(k)$, also replacing $A$ with $A'$.
 
 Considering the tower of small extensions $R/\m(R)^{n+1} \to R/\m(R)^n$, it follows from Lemma \ref{2fibrnlemma} that the simplicial functor $\uline{\Alg}^{\ctr,c}_{\cP}(R) \to \uline{\Alg}^{c}_{\cP}(k)$ is a 2-fibration of simplicial categories, so the homotopy fibre over $A'$ is simply given by the categorical 2-fibre, since that corresponds to the fibre over a fibration in the model structure of \cite[Theorem 1.1]{bergner}.
 
 In other words, our simplicial category of interest consists of pairs $(\tilde{A}, \theta)$ with  $\tilde{A} \in \Alg^{\ctr,c}_{\cP}(R)$ and $\theta \co \tilde{A}\ten_Rk \to A'$ a fixed isomorphism, together with the obvious simplicial morphisms. 
 Any such object  is isomorphic to one of the form $(A'\ten R, \tilde{d})$ for a square-zero differential $\tilde{d}$ of chain degree $-1$ with $\tilde{d} \equiv d_{A'} \mod \m(R)$. Equivalently,  $\tilde{d} = d_{A'} +\omega$ with $\omega \in \mc(\DDer_{\cP}(A',A')\ten\m(R))$.

 The simplicial set $I(\omega_1,\omega_2)$ of morphisms $( A'\ten R, d_{A'} + \omega_1) \to ( A'\ten R, d_{A'} + \omega_2) $ between these which map to the identity in   $\uline{\Alg}^{\ctr,c}_{\cP}(k)$ is then given in level $n$ by
 \[
  \{g \in \Gg(\DDer_{\cP}(A',A'\ten \Omega^{\bt}(\Delta^n))\ten \m(R)) ~:~ g  \star \omega_1 = \omega_2 \},
 \]
under the exponentiation correspondence between nilpotent derivations and infinitesimal automorphisms. The set of morphisms $\omega \to \omega'$ in $\DDel(\DDer_{\cP}(A',A'), R)$ (Definition \ref{ddeldef}) is the corresponding expression with $\DDer_{\cP}(A',A')\ten \Omega^{\bt}(\Delta^n) $ in place of $\DDer_{\cP}(A',A'\ten \Omega^{\bt}(\Delta^n))$. 
This gives us
a functor from $\DDel(\DDer_{\cP}(A',A'), R)$ to our simplicial category, and since it gives an equivalence in degree $0$, we can see that it is a weak equivalence
because
\[
 \pi_i\uline{\Aut}(\omega) \cong \H^{-i}(\DDer_{\cP}(A',A')\ten \m(R), d+ [\omega,-]) \cong \pi_iI(\omega,\omega).\qedhere
\]

 \end{proof}

\begin{remarks}\label{1proxrmk1}
 Since $\cD_{dg}(R)$ is a full dg subcategory of  $\cD_{dg}^{\ctr}(R)$, Theorem \ref{cofalgdefthm} implies that the functor of contraderived deformations is  universal among the functors governed by DGLAs and lying  under the functor of derived deformations. In the problematic terminology\footnote{see \cite[Remark 4.2]{utrecht}} of \cite[Definition 5.1.5]{lurieDAG10}, this says contraderived deformations are the formal moduli problem associated to the $1$-proximate formal moduli problem of derived deformations.
 
 Note that, as is usually the case with derived deformation problems, it was easier for us just to construct the DGLA and prove the equivalence than it would have been to apply \cite[Theorem 4.14 and Corollary 4.57]{ddt1} (later recovered as \cite[Theorem 0.0.13]{lurieDAG10}) to infer existence of a governing DGLA indirectly.
\end{remarks}

The following slightly generalises \cite[Theorem 2.3.4]{hinichDefsHtpyAlg}, which requires both $\cP$ and $A$ to be concentrated in non-negative chain degrees.
\begin{corollary}\label{cofalgdefcor}
 Given a $k$-linear dg operad $\cP$ and a $\cP$-algebra $A$ satisfying $\H_iA=0$ for all $i \ll 0$,
  with cofibrant replacement $A'$, the derived Deligne groupoid $\DDel(\DDer_{\cP}(A',A'), R)$ is canonically quasi-equivalent to the simplicial category of derived deformations of $A$ over $R$, for all $R \in dg_+\Art_k$.
\end{corollary}
\begin{proof}
In $\cD^{\ten}_{dg}(R)$, quasi-isomorphisms are precisely mod-quasi-isomorphisms, so $\oL^{\cW} \Alg_{\cP}(R)$ is a full $\infty$-subcategory of $\oL^{\bar{\cW}} \Alg^{\ctr}_{\cP}(R)$.
 Lemma \ref{connectivitylemma} implies that it contains all objects lying over $A \in \oL^{\cW}\Alg_{\cP}(k)$, so the map from the simplicial category of derived deformations of $A$ to that of contraderived deformations is a quasi-equivalence in this case. The result then follows immediately from Theorem \ref{cofalgdefthm}.
\end{proof}


\section{Deformations as mapping spaces of coloured dg (pr)operads}\label{propsn}

\subsection{Twisting and the hat construction} 
 
 
 For a coloured dg (pr)operad $\cP$, we have a dg associative algebra $\cP(x;x)$ for every colour $x \in \Ob \cP$, with multiplication given by composition.  The following generalises \cite[5.1]{ChuangLazarevComb} to incorporate multiple objects and allow for properads.
 \begin{definition}
  Given a coloured dg (pr)operad or dg category $\cP$ and a collection of Maurer--Cartan elements $\uline{\omega}= \{\omega_x \in \mc(\cP(x;x))\}_{x \in \Ob \cP}$, define the \emph{twist} $\cP^{\uline{\omega}}$ of $\cP$ by $\uline{\omega}$ to be given by $\cP$ as a coloured graded (pr)operad or category, but with differential $d^{\uline{\omega}}$ given on $\cP(x_1, \ldots, x_r;y_1, \ldots, y_s)$ by
  \[
 d^{\uline{\omega}}(a):=  d a + \sum_{j=1}^s\omega_{y_j} \circ_j a - (-1)^{\deg a} \sum_{i=1}^r a \circ_i \omega_{x_i}
  \]
\end{definition}

 \begin{definition}
  Given a coloured dg (pr)operad or dg category $\cQ$, a set $S$ and a function $g \co S \to \Ob \cQ$, define $g^{-1}\cQ$ to be the dg (pr)operad or dg category on colours $S$ with multimorphisms $(g^{-1}\cQ)(s_1, \ldots, s_m;t_1, \ldots, t_n):= \cQ(g(s_1), \ldots, g(s_m);g(t_1), \ldots, g(t_n))$.
   \end{definition}

 The following definition generalises \cite{ChuangLazarevComb} to coloured dg (pr)operads.
\begin{definition}
Define the endofunctor $\cP \leadsto \cP\<m\>$ of the category of coloured dg (pr)operads or of dg categories by freely adjoining elements $m_x \in \mc(\cP\<m\>(x;x))$ for all $x \in \Ob \cP$.

Then define the endofunctor $\cP \leadsto \widehat{\cP}$ by setting $\widehat{\cP}$ to be the twist $\cP\<m\>^m$.
 \end{definition}
 
 Thus giving a morphism $\cP\<m\> \to \cQ$ of coloured dg (pr)operads amounts to giving a morphism $f \co \cP \to \cQ$ together with elements $\omega_x \in \mc(\cQ(f(x);f(x)))$ for all $x \in \Ob \cP$.
 
 Likewise, a morphism $\widehat{\cP} \to \cQ$ amounts to giving a morphism $f_0 \co \Ob \cP \to \Ob \cQ$, Maurer--Cartan elements $\omega_x \in \mc(\cQ(f_0(x);f_0(x)))$ for all $x \in \Ob \cP$ and a morphism $\cP \to (f_0^{-1}\cQ)^{\uline{\omega}}$ of coloured dg (pr)operads on fixed colours $\Ob \cP$.
 
 \begin{definition}
  Define  the endofunctor $\cP \leadsto \MC\cP$ of the category of coloured dg (pr)operads or of dg categories to be the right adjoint to $\widehat{(-)}$. 
  
  Explicitly, colours of $  \MC\cP$ are pairs $(x, \omega)$ with $x \in \Ob \cP$ and $\omega \in \mc(\cP(x;x))$. For the projection map $\pr_1 \co \Ob \MC\cP \to \Ob \cP$, the dg (pr)operad $  \MC\cP$ is then given by the twist $(\pr_1^{-1}\cP)^{\pr_2}$, where $\pr_2(x, \omega)=\omega$.
   \end{definition}

   \subsection{Relation to contraderived categories}
   
   \begin{definition}
   Given a $k$-linear coloured dg (pr)operad or dg category $\cQ$  and $R =\{R(i)\}_i \in \pro(dg\Art_k)$, define the coloured dg (pr)operad $\cQ\hten R$ to have the same colours as $\cQ$ and multimorphisms $(\cQ\hten R)(x_1,\ldots,x_r;y_1, \ldots, y_s)= \Lim_i  \cQ(x_1,\ldots,x_r;y_1, \ldots, y_s)\ten R(i)$, with the obvious compositions.
   \end{definition}

   \begin{proposition}\label{Dctrcheckprop}
  For $R\in \pro(dg\Art_k)$, the coloured dg operad $\cD^{\ctr,\ten}_{dg}(R)$ (resp. coloured dg properad $\cD^{\ctr,\ten,\ten}_{dg}(R)$, resp. dg category $\cD^{\ctr}_{dg}(R)$)  of contraderived projective modules is canonically equivalent to the fibre product of the diagram
  \begin{align*}
   \MC(\cD^{\ten}_{dg}(k)\hten R)\to \MC(\cD_{dg}^{\ten}(k)) \la \cD_{dg}^{\ten}(k), \text{ resp.}\\
   \MC(\cD^{\ten,\ten}_{dg}(k)\hten R)\to \MC(\cD_{dg}^{\ten,\ten}(k)) \la \cD_{dg}^{\ten,\ten}(k), \text{ resp.}\\
    \MC(\cD_{dg}(k)\hten R)\to \MC(\cD_{dg}(k)) \la \cD_{dg}(k).
  \end{align*}
\end{proposition}
\begin{proof}
 A colour of the fibre product is a pair $(V, \omega)$ with $V \in \cD_{dg}(k)$ and 
 \[
 \omega \in \mc(\ker(\EEnd_k(V)\hten R \to \EEnd_k(V))) \cong \mc(\EEnd_k(V)\hten \m(R)).
 \]
 To such a pair, we thus associate the object $M(V,\omega) := (V\ten R, d +\omega)$ of $\cD^{\ctr}_{dg}(R)$, with the Maurer--Cartan equation ensuring closure of the differential. 
 
Since the right-hand map in the fibre product diagram is fully faithful, the complex of multimorphisms in the fibre product from $((V_1, \omega_1), \ldots, (V_r, \omega_r))$ to $((W, \nu_1), \ldots, (W_s, \nu_s))$ is just given by the left-hand term, i.e.
\[
(\HHom_k(V_1\ten \ldots \ten V_r,W_1\ten\ldots \ten W_s)\hten R, d^{\omega, \nu}).
\]
This is canonically isomorphic (compatibly with compositions) to
\[
\HHom_R(M(V_1,\omega_1)\ten_R \ldots \ten_R M(V_r, \omega_r),M(W_1, \nu_1) \ten_R \ldots \ten_R M(W_s, \nu_s)),
 \]
giving us a fully faithful dg multifunctor $M(-)$. As in the proof of Theorem \ref{cofalgdefthm}, it follows easily from the definitions that
 every projective contramodule is isomorphic to one of the form $M(V, \omega)$, giving essential surjectivity of $M(-)$.
\end{proof}

 In fact, more is true: for the  dg category $\cD^{\ten}_g(k)$ of graded $k$-vector spaces (with $d=0$), we have $\cD^{\ctr}_{dg}(R) \simeq \MC(\cD_{g}(k)\hten R)$, and similarly for the associated coloured dg (pr)operads.

 \subsection{Deforming morphisms of coloured dg (pr)operads}
 
 Turning Proposition \ref{Dctrcheckprop} into a definition and generalising gives the following.
 \begin{definition}\label{Qctrdef}
  Given a $k$-linear coloured dg (pr)operad or dg category $\cQ$ and $R \in \pro(\Art_k)$, define the coloured dg (pr)operad or dg category of $R$-linear contraderived deformations of $\cQ$ by
 \[
 \cQ^{\ctr}(R) :=  \MC(\cQ\hten R)\by_{  \MC\cQ}\cQ.
 \]
  \end{definition}

 The following is an immediate consequence of the adjunction $\widehat{(-)} \dashv \MC(-)$:
 \begin{lemma}\label{ctrchecklemma}
 In the category of small $k$-linear coloured dg (pr)operads or of small dg categories, 
 we have
 \[
  \Hom(\cP,\cQ^{\ctr}(R)) \cong  \Hom(\widehat{\cP},\cQ\hten R)\by_{\Hom(\widehat{\cP},\cQ)}\Hom(\cP,\cQ). 
 \]
 \end{lemma}

 \begin{definition}
  Given a co-augmented dg coproperad $\C= \bar{\C} \oplus I$, define the reduced cobar construction $\bar{\Omega}(\C)$ to be the dg properad $\Omega(\bar{\C})$ of \cite[\S 3.6]{MerkulovValletteDefThRepsProps}, given by the free graded properad generated by the desuspension of $\bar{\C}$, with differential combining that on $\bar{\C}$ with its partial coproduct.
 \end{definition}
Beware that $\bar{\Omega}(\C)$ itself is also denoted $\Omega(\C)$ in \cite[\S 3.6]{MerkulovValletteDefThRepsProps}, but that would more logically mean $\bar{\Omega}(\C \oplus I)$, where $-\oplus I$ denotes the formal addition of a co-identity. We cannot afford to confuse the constructions for reasons which will become immediately apparent. 
 
 The following straightforward observation generalises \cite[Remark 5.4]{ChuangLazarevComb} to properads.
 \begin{lemma}\label{hatOmegalemma}
 There is a canonical isomorphism $\widehat{\bar{\Omega}(\C)}\cong \bar{\Omega}(\C \oplus I)$ of dg properads for any co-augmented dg coproperad $\C$.
\end{lemma}
 
 \begin{definition}\label{DDergammadef}
  Take a co-augmented dg coproperad $\C=\bar{\C} \oplus I$, a coloured dg properad $\cQ$ and a  morphism $\gamma \co \bar{\Omega}\C \to \cQ$, with $\ast$  the unique colour of $\bar{\Omega}\C$. 
  
  Define the DGLA $\HHom^{\bS}(\bar{\C},\cQ|_{\gamma(\ast)})^{\gamma}$ to be the complex of  $\bS$-bimodule homomorphisms from $\C$ to  $\cQ|_{\gamma(\ast)}$, equipped with the  convolution Lie bracket of \cite[\S 2.4]{MerkulovValletteDefThRepsProps}, and with twisted differential $d +[\gamma,-]$, 
  
  Then define $\DDer_{\bar{\Omega}\C}(\gamma) :=  \HHom^{\bS}(\bar{\C} \oplus I,\cQ|_{\gamma(\ast)})^{\gamma}$, as considered in \cite[\S 8.2, Remark]{MerkulovValletteDefThRepsProps}.
 \end{definition}
By  \cite[Theorem 69]{MerkulovValletteDefThRepsProps}, $\HHom^{\bS}(\bar{\C},\cQ|_{\gamma(\ast)})^{\gamma}$ is isomorphic to the complex  $\DDer(\bar{\Omega}\C, \cQ|_{\gamma(\ast)})$  of properad derivations from $\bar{\Omega}\C$ to  $\cQ|_{\gamma(\ast)}$, where the module structure on $\cQ|_{\gamma(\ast)}$ comes from $\gamma$. Thus by Lemma \ref{hatOmegalemma}, $\DDer_{\bar{\Omega}\C}(\gamma) $ is isomorphic to $\DDer(\widehat{\bar{\Omega}\C}, \cQ|_{\gamma(\ast)})$.

 \begin{remark}[Comparison of DGLAs]\label{compDGLArmk}
  Observe that if $\cQ= \cD_{dg}(k)^{\ten,\ten}$ and $\C$ is the bar construction of an augmented dg properad $\cP$, then $\gamma\co \bar{\Omega}\C \to  \cD_{dg}(k)^{\ten,\ten}$ corresponds to a $\cP_{\infty}$-algebra $A$ and $\DDer_{\bar{\Omega}\C}(\gamma)$ is isomorphic to the DGLA  of $\cP_{\infty}$ $\infty$-derivations of $A$ (in the terminology of \cite[\S 10.2.2]{LodayValletteOperads}). By contrast, $\DDer(\bar{\Omega}\C, \cQ)$  just consists of $\cP_{\infty}$ $\infty$-derivations whose linear term is $0$, corresponding to deformations fixing the underlying chain complex. 

If $\cP$ is a dg operad and $A$ is in fact a $\cP$-algebra, then our DGLA $\DDer_{\bar{\Omega}\C}(\gamma)$ is quasi-isomorphic to the DGLA $\DDer_{\cP}(\tilde{A},\tilde{A})$ of Theorem \ref{cofalgdefthm}. To see this, observe that the universal twisting morphism $\alpha \co \bar{\Omega}(\C) \to \cP$ gives  a bar-cobar adjunction $\Omega_{\alpha} \dashv \b_{\alpha}$ as in \cite[\S 11.2]{LodayValletteOperads}, with isomorphisms
\[
 \HHom^{\bS}(\C, \cD_{dg}(k)^{\ten}|_{\gamma(\ast)})^{\gamma} \cong   \HHom^{\bS}(\C, \EEnd^{\ten}_k(A))^{m_A} \cong 
 \Co\DDer_{\C}(\b_{\alpha}A,\b_{\alpha}A)
\]
of DGLAs. The functor $\Omega_{\alpha}$ gives a quasi-isomorphism from this to  $\DDer_{\cP}(\tilde{A}, \tilde{A})$, 
for the cofibrant replacement
$\tilde{A}:=\Omega_{\alpha}\b_{\alpha}A$  of $A$. 
  \end{remark}

As a consequence of Lemma \ref{hatOmegalemma} and the proof of \cite[Proposition 17]{MerkulovValletteDefThRepsProps}, we have:
\begin{lemma}\label{hatMClemma}
Given  a co-augmented dg coproperad $\C$ 
and a small coloured dg properad $\cQ$,  the set of dg properad morphisms from $\widehat{\bar{\Omega}(\C)}$ to $\cQ$ is isomorphic to 
\[
 \prod_{x \in \Ob \cQ}\mc(\HHom^{\bS}(\C,\cQ|_x)).
\] 
\end{lemma}
Note that the formula features $\C$, not $\bar{\C}$, since the source is $\widehat{\bar{\Omega}(\C)}$ instead of $\bar{\Omega}(\C)$. 

Combining Lemmas \ref{ctrchecklemma} and \ref{hatMClemma}, we have:
\begin{proposition}\label{checkMCprop}
 Given a co-augmented dg coproperad $\C$ and  a coloured dg properad $\cQ$, there is a canonical isomorphism
 \begin{align*}
  \Hom(\bar{\Omega}(\C),\cQ^{\ctr}(R)) &\cong \prod_{x \in \Ob \cQ} \mc(\HHom^{\bS}(\C,\cQ|_x)\hten R  )\by_{\mc(\HHom^{\bS}(\C,\cQ|_x))}\mc(\HHom^{\bS}(\bar{\C},\cQ|_x))\\
&\cong \prod_{\gamma \co \bar{\Omega}(\C) \to \cQ} \mc(\DDer_{\bar{\Omega}\C}(\gamma) \hten \m(R))  
  \end{align*} 
  for all $R \in \pro(dg\Art)$.
\end{proposition}

\subsubsection{The contraderived construction as a Quillen functor}

Following standard sign conventions for shifts, given a dg category $\cQ$,  we have a  dg associative algebra $(\cQ(y;y) \oplus \cQ(x;y)[1] \oplus  \cQ(x;x))$ with differential $d(\alpha,\theta, \beta)= ([d, \alpha], -[d, \theta], [d, \beta])$ and multiplication
 \[
  (\alpha,\theta,\beta) \cdot (\alpha',\theta',\beta'):= (\alpha \circ \alpha', (-1)^{\deg \alpha} \alpha \circ \theta' + \theta \circ \beta',\beta \circ \beta').
 \]
 
\begin{lemma}\label{morphismMClemma} Given $R \in dg\Art_k$, a  $k$-linear dg category $\cQ$ and a morphism $f_0 \co x \to y$ in the underlying category $\z^0\cQ$, the set of morphisms in $\z^0\cQ^{\ctr}(R)$ lying over $f_0$ is isomorphic to the Maurer--Cartan set $\mc(  (\cQ(y;y) \oplus \cQ(x;y)[1] \oplus  \cQ(x;x))^{f_0} ,R)$ of the twisting by $f_0$. 
 \end{lemma}
\begin{proof}
An element of $\mc( (\cQ(y;y) \oplus \cQ(x;y)[1] \oplus  \cQ(x;x))^{f_0},R)$ is given by $\alpha \in \mc(\cQ(y;y),R)$, $\beta \in \mc(\cQ(x;x),R)$ and $\theta \in (\cQ(x;y)\ten\m(R))^0$ satisfying 
\begin{align*}
&-[d, f_0 + \theta] - \alpha \circ (f_0 + \theta) + (f_0 + \theta) \circ \beta = 0, \quad\text{ i.e.}\\
&(f_0+\theta) \circ (d + \beta) = (d + \alpha) \circ (f_0 + \theta).
\end{align*}
 The latter is precisely the condition for $f_0 + \theta$ to be a morphism from $(x,\beta)$ to $(y,\alpha)$ in $\z^0\cQ^{\ctr}(R)$ lying over $f_0$.
\end{proof}

By \cite[Theorems 3.1 and 3.2]{hinstack} and \cite[Proposition \ref{ddt1-dgspmodel}]{ddt1},
there is a fibrantly cogenerated closed model structure on $\pro(dg\Art_k)$ with cogenerating fibrations given by small extensions in $dg\Art_k$ and cogenerating trivial fibrations given by small extensions with acyclic kernel. The relevant notion of weak equivalence is stronger than quasi-isomorphism (making this a model structure of the second kind in the sense of \cite{positselskiDerivedCategories}), and corresponds to quasi-isomorphism of the Koszul dual DGLAs, or to a homotopy lifting property with respect to quasi-free algebras as in  \cite[Definition \ref{ddt1-dgzweakdef}]{ddt1}.

\begin{proposition}\label{ctrrightQprop}
For any small dg (pr)operad $\cQ$, the functor $\cQ^{\ctr}$ from $\pro(dg\Art_k)$  to the category of small coloured $k$-linear  dg (pr)operads over $\cQ$ is right Quillen.
\end{proposition}
\begin{proof}
The characterisation of  $\cQ^{\ctr}$ in Lemma \ref{ctrchecklemma} ensures that it preserves all limits when the codomain of the functor is taken to be the slice category over $\cQ$. Since $dg\Art_k$ is an Artinian category, the pro-representability theorem of \cite[Corollary to Proposition 3.1]{descent} thus implies that $\cQ^{\ctr}$ has a left adjoint.

It remains to show that $\cQ^{\ctr}$ preserves fibrations and trivial fibrations, and it suffices to check these conditions on the cogenerating morphisms, since limits are preserved. Given a small extension $R \to R/I$ in $dg\Art_k$, we have short exact sequences
\[
0 \to \cQ(\uline{x}; \uline{y})\ten I \to  \cQ^{\ctr}(R)((\uline{x}, \uline{\omega}); (\uline{y}, \uline{\nu})) \to \cQ^{\ctr}(R/I)((\uline{x}, \uline{\bar{\omega}}); (\uline{y}, \uline{\bar{\nu}})) \to 0
\]
for all colours $(x_i,\omega_i)$ and $(y_j, \nu_j)$ of $\cQ^{\ctr}(R)$, with respective images $(x_i,\bar{\omega}_i)$ and $(y_j, \bar{\nu}_j)$ in $\cQ^{\ctr}(R/I)$. We thus have surjectivity on complexes of multimorphisms, with the maps moreover being quasi-isomorphisms when $I$ is acyclic.

The standard obstruction argument  of e.g. \cite[Lemma \ref{utrecht-obsdglalemma}]{utrecht} shows that the obstruction to lifting an element from 
$\mc(\cQ(x;x)\ten \m(R)/I)$ 
to $\mc(\cQ(x;x)\ten \m(R)) $
lies in  $\H^{2}(\cQ(x;x)\ten I)$, 
which  gives surjectivity of $\cQ^{\ctr}(R) \to \cQ^{\ctr}(R/I)$ on objects when $I$ is acyclic.

Now consider a homotopy equivalence $f_0 \in \z^0\cQ(x;y)$.
The complex underlying the  DGLA $L:= (\cQ(y;y) \oplus \cQ(x;y)[1] \oplus  \cQ(x;x))^{f_0}$ from Lemma \ref{morphismMClemma} is the cocone of $(-f_0^*, f_{0*}) \co  \cQ(y;y) \oplus \cQ(x;x) \to \cQ(x;y)$. Since $f_0$ is a homotopy equivalence, the maps $f_0^*$ and $f_{0*}$ are both quasi-isomorphisms,  so the projection maps $L \to \cQ(x;x)$ and $L \to \cQ(y;y)$ are both  quasi-isomorphisms of DGLAs. For any small extension $R \to R/I$, the central extension 
\[
L\ten \m(R) \to (\cQ(x;x)\ten \m(R))\by_{(\cQ(x;x)\ten \m(R/I))}(L\ten \m(R/I))
\]
of DGLAs thus has acyclic kernel (and similarly for $y$ replacing $x$), so is surjective on Maurer--Cartan elements by \cite[Lemma \ref{utrecht-obsdglalemma}]{utrecht}. 

Lemma \ref{morphismMClemma} allows us to interpret Maurer--Cartan elements on the right as the data of an element $(x,\omega) \in \Ob \cQ^{\ctr}(R)$ together with a morphism  $f \in \z^0\cQ^{\ctr}(R/I) ((x, \bar{\omega});(y,\nu))$ lifting $f_0$. Surjectivity then says that this lifts to a morphism $\tilde{f}$ in $\z^0\cQ^{\ctr}(R)$ from $(x, \omega)$ to some lift $(y, \tilde{\nu})$ of $(y,\nu)$. 

To complete the  argument that $\H^0\cQ^{\ctr}(R) \to \H^0\cQ^{\ctr}(R/I)$ induces an isofibration of the underlying categories, 
it remains to show that the morphism $\tilde{f}$ is a homotopy equivalence whenever $f$ is so (which will thus be whenever $f_0$ is so, by induction). The same argument in reverse allows us to lift the homotopy inverse $g$ of $f$ to a morphism $\tilde{g}$, and we can also choose arbitrary lifts of the homotopies $f\circ g \sim \id$ and $g\circ f \sim \id$. Lifts of the identity morphism $\id_{(y,\nu)}$ are of the form $\id + \z^0(\cQ(y;y)\ten I)$, and since $I$ is square-zero, these are all isomorphisms (with $(\id+u)^{-1}= \id -u$). Thus $\tilde{f}\circ \tilde{g}$ and $\tilde{g}\circ \tilde{f}$ are both homotopic to isomorphisms, so  their homotopy classes are both isomorphisms in $\H^0\cQ^{\ctr}(R)$, meaning that   $[\tilde{f}]$ must also be so.
\end{proof}
 
 \subsubsection{Derived derivations govern the mapping space} 
 
 \begin{definition}
  As in \cite[\S 1.3]{MerkulovValletteDefThRepsProps}, we say that a co-augmented dg coproperad $\C = \bar{\C} \oplus I$  is \emph{connected} if the non-primitive part $\bar{\Delta} \co \C \to \bar{\C}\boxtimes_c\bar{\C}$ of the coproduct is  ind-conilpotent, i.e. if $\C$ has  an exhaustive  filtration
  \[
   F_r\C:= \ker(\bar{\Delta}^{(r)} \co \C \to \underbrace{\bar{\C}  \boxtimes_c\bar{\C}\boxtimes_c \ldots \boxtimes_c\bar{\C}}_r).
 \]
 \end{definition}

\begin{lemma}\label{barOmegafiblemma}
 For any connected $k$-linear dg coproperad $\C$, the dg properad $\bar{\Omega}\C$ is cofibrant.
\end{lemma}
\begin{proof}
For any $k$-linear dg properad $\cP$, the set of dg properad morphisms from  $\bar{\Omega}\C$ to $\cP$ is isomorphic to $\mc(\HHom^{\bS}(\bar{\C},\cP))$, as in the proof of \cite[Proposition 17]{MerkulovValletteDefThRepsProps}. Now the definition of the filtration $F$ ensures that for any trivial fibration $\cP \to \cP'$, the morphism
\[
 \HHom^{\bS}(F_{r+1}\bar{\C},\cP) \to \HHom^{\bS}(F_{r+1}\bar{\C},\cP')\by_{\HHom^{\bS}(F_{r}\bar{\C},\cP')}\HHom^{\bS}(F_{r}\bar{\C},\cP)
\]
is an acyclic central extension of DGLAs, so the standard obstruction argument  of e.g. \cite[Lemma \ref{utrecht-obsdglalemma}]{utrecht} implies that there is no obstruction to lifting a Maurer--Cartan element. Thus  the morphism $\bar{\Omega}(F_r\C) \to \bar{\Omega}(F_{r+1}\C)$ is a cofibration for all $r$, with $I= \bar{\Omega}(F_{-1}\C)$ and $\bar{\Omega}\C= \LLim_r \bar{\Omega}(F_r\C)$.
\end{proof}

 The following theorem shows that the DGLA of $\bar{\Omega}\C$-algebra derivations from Definition \ref{DDergammadef} governs deformations of an algebraic structure, regarded as a morphism of coloured dg (pr)operads.

 \begin{theorem}\label{mainthm}
  Given  a coloured $k$-linear dg (pr)operad $\cQ$, a connected dg co(pr)operad $\C$   and a morphism $\gamma \co \bar{\Omega}\C \to \cQ$,
  there is a natural zigzag of weak equivalences
  \[
\oR\map(\bar{\Omega}\C, \cQ^{\ctr}(R))\by^h_{\oR\map(\bar{\Omega}\C, \cQ) } \{\gamma\} \simeq   \mmc(   \DDer_{\bar{\Omega}\C}(\gamma),R)
  \]
of simplicial sets, natural in $R \in \pro(dg\Art_k)$, where the mapping spaces on the left are taken in the $\infty$-category of coloured dg (pr)operads localised at quasi-equivalences.
 \end{theorem}
\begin{proof}
As in  \cite[Theorem \ref{ddt1-mcequiv}]{ddt1} or \cite[Theorem 3.2]{hinstack}, there is a contravariant Quillen equivalence $\b^*$ from the category of $k$-linear DGLAs to $\pro(dg\Art_k)$. By Proposition \ref{checkMCprop}, we then have 
\begin{align*}
 \Hom(\bar{\Omega}\C\xra{\gamma} \cQ, \cQ^{\ctr}(R) \to \cQ)_{\cQ} &\cong \mc(\DDer_{\bar{\Omega}\C}(\gamma)\hten \m(R))\\
 &\cong \Hom
 ( \b^*\DDer_{\bar{\Omega}\C}(\gamma),R).
\end{align*}

Since right Quillen functors preserve simplicial Reedy framings \cite[\S 5.4]{hovey} and $\bar{\Omega}\C$ is cofibrant,  Proposition \ref{ctrrightQprop} implies that for any such framing $\widehat{R}(\ast)$ of $R \in \pro(dg\Art_k)$, we have
\[
\oR\map(\bar{\Omega}\C\xra{\gamma} \cQ, \cQ^{\ctr}(R) \to \cQ)_{\cQ} \simeq   \Hom(\bar{\Omega}\C\xra{\gamma} \cQ, \cQ^{\ctr}(\widehat{R}(\ast)) \to \cQ)_{\cQ},
\]
which in turn is equivalent to
\[
 \Hom( \b^*\DDer_{\bar{\Omega}\C}(\gamma),\widehat{R}(\ast)) \simeq \oR\map_{\pro(dg\Art_k)}( \b^*\DDer_{\bar{\Omega}\C}(\gamma),R)
\]
under the equivalences above. Since $\b^*$ is  Quillen and  $L\ten \Omega^{\bt}(\Delta^{\ast})$ gives a simplicial Reedy framing for any DGLA $L$, this is equivalent to
\[
 \Hom( \b^*(\DDer_{\bar{\Omega}\C}(\gamma)\ten \Omega^{\bt}(\Delta^{\ast})) ,R)= \mmc(\DDer_{\bar{\Omega}\C}(\gamma),R). \qedhere
\]
\end{proof}

In order to avoid category theory with universes, for any cardinal $\kappa$ consider $\cD_{dg,\kappa}^{\ctr,\ten,(\ten)}(R) \subset \cD_{dg}^{\ctr,\ten,(\ten)}(R)$ consisting of complexes $M$ with $\sum_i \rk_R M_i < \kappa$, and similarly for $\cD_{dg,\kappa}^{\ten,(\ten)}(R)$. Corollary \ref{maincor} will imply that increasing $\kappa$ does not affect the following space.
\begin{definition}
 Given a $k$-linear dg (pr)operad $\cP$, a $\cP$-algebra  $\gamma \co \cP \to \cD_{dg,\kappa}^{\ten,(\ten)}(k)$ and $R \in dg\Art_k$, define the space $\Def_{\cP_{\infty}}(\gamma,R)$ of derived deformations of $\gamma$ to be the homotopy fibre product
 \[
  \oR\map(\cP, \cD^{\ten,(\ten)}_{dg,\kappa}(R))\by^h_{\oR\map(\cP, \cD^{\ten,(\ten)}_{dg,\kappa}(k)  ) } \{\gamma\},
 \]
where mapping spaces are taken in the $\infty$-category of coloured dg (pr)operads localised at quasi-equivalences.
  
 Similarly, for  $R \in \pro(dg\Art_k)$, define the space $\Def_{\cP_{\infty}}^{\ctr}(\gamma,R)$ of contraderived deformations of $\gamma$ to be the homotopy fibre product
 \[
  \oR\map(\cP, \cD^{\ctr,\ten,(\ten)}_{dg,\kappa}(R))\by^h_{\oR\map(\cP, \cD^{\ten,(\ten)}_{dg,\kappa}(k)  ) } \{\gamma\}.
 \]
  \end{definition}

 \begin{corollary}\label{maincor}
  Given a dg (pr)operad $\cP$, a $\cP$-algebra  $\gamma \co \cP \to \cD_{dg}(k)$ and a quasi-isomorphism $\bar{\Omega}\C \to \cP$ for some connected dg co-operad $\C$, there is a canonical equivalence
  \[
  \Def_{\cP_{\infty}}^{\ctr}(\gamma,R) \simeq \mmc(   \DDer_{\bar{\Omega}\C}(\gamma),R),
  \]
  of simplicial sets, 
functorial in $R \in \pro(dg\Art)$.

If $R \in dg_+\Art$ and the complex $\gamma(\ast)$ underlying $\gamma$ is homologically bounded below, then the canonical map
\[
 \Def_{\cP_{\infty}}(\gamma,R) \to \Def_{\cP_{\infty}}^{\ctr}(\gamma,R) 
\]
is also an equivalence.  
 \end{corollary}
\begin{proof}
 The first statement follows from Theorem \ref{mainthm} applied to the small dg (pr)operad $\cQ= \cD_{dg,\kappa}^{\ten,(\ten)}(k)$ via the equivalence of Proposition \ref{Dctrcheckprop}. The second statement  follows because Lemma \ref{connectivitylemma}  implies that $\cD_{dg}(R)\by^h_{\cD_{dg}(k)}\{\gamma(\ast)\} \simeq   \cD^{\ctr}_{dg}(R)\by^h_{\cD_{dg}(k)}\{\gamma(\ast)\}$ under those additional hypotheses. 
\end{proof}

 \begin{remark}\label{proximateDdgrmk}
 Since the derived dg category forms a full subfunctor of the contraderived dg category, we can conclude along the lines of  \cite[\S 5.1]{lurieDAG10} that representability of $\Def_{\cP_{\infty}}^{\ctr}(\gamma,-)$ by a DGLA makes it the universal functor under  $\Def_{\cP_{\infty}}(\gamma,-)$ preserving homotopy pullbacks along small extensions, and the same is true if we replace  $\Def_{\cP_{\infty}}(\gamma,-)$ with the component of the trivial deformation. 
 
 Allowing $\cP$ and $\gamma$ to vary and $\cP$ to incorporate colours, Proposition \ref{ctrrightQprop} implies by the same reasoning that $\cD_{dg}^{\ctr,\ten,(\ten)}(-)$ is the universal such functor under either  $\cD_{dg}^{\ten,(\ten)}(-)$ or $ \cD_{dg}^{\ten,(\ten)}(k)\ten-$, and likewise for $\cQ^{\ctr}$ under $\cQ\hten-$. In other words, among constructions with well-behaved deformation theory, the contraderived dg category is the closest possible approximation to the derived dg category.
 \end{remark}

\bibliographystyle{alphanum}
\bibliography{references.bib}

\end{document}